\documentclass[letterpaper, 10 pt, conference]{ieeeconf}
\usepackage{amssymb,tikz, graphicx, subcaption,epstopdf}
\usepackage{amsfonts}
\usepackage{amsmath,mathrsfs}
\usepackage{color}
\usepackage{url}
\usepackage{steinmetz}
\IEEEoverridecommandlockouts
\input{tcilatex}

\title{\LARGE \bf{Comparison of Centralized and Decentralized Approaches in Cooperative Coverage Problems with Energy-Constrained Agents}}
\author{Xiangyu Meng, Xinmiao Sun, Christos G. Cassandras, and Kaiyuan Xu
\thanks{This work was supported in part by NSF
under grants ECCS-1509084, DMS-1664644, and CNS- 1645681, by AFOSR under grant
FA9550-19-1-0158, by ARPA-E's NEXTCAR program under grant
DE-AR0000796, by the MathWorks, by the Fundamental Research Funds for the Central Universities under Grant FRF-TP-19-034A1, by Guangdong Basic and Applied Basic Research Foundation under Grant 2019A1515111039, and by China Postdoctoral Science Foundation funded project under Grant 2020M670136.}
\thanks{X. Meng is
with the Division of Electrical and Computer Engineering, Louisiana State
University, Baton Rouge, LA, 70803 USA {\tt\small xmeng5@lsu.edu}}
\thanks{X. Sun is
with the School of Automation and Electrical Engineering, Shunde Graduate School and Key Laboratory of Knowledge Automation for Industrial Processes, University of Science and Technology Beijing, 100083, China {\tt\small xmsun@ustb.edu.cn}}
\thanks{C. G.
Cassandras and K. Xu are with the Division of Systems Engineering and Center for
Information and Systems Engineering, Boston University, Brookline, MA 02446
USA {\tt\small cgc@bu.edu; xky@bu.edu}}
}
\graphicspath{{Figure/}}
\DeclareGraphicsExtensions{.pdf,.eps}

\begin{document}

\maketitle

\begin{abstract}
A multi-agent coverage problem is considered with energy-constrained agents. The objective of this paper is to compare the coverage performance between centralized and decentralized approaches.  To this end, a near-optimal centralized coverage control method is developed under energy depletion and repletion constraints. The optimal coverage formation corresponds to the locations of agents where the coverage performance is maximized. The optimal charging formation corresponds to the locations of agents with one agent fixed at the charging station and the remaining agents maximizing the coverage performance. We control the behavior of this cooperative multi-agent system by switching between the optimal coverage formation and the optimal charging formation. Finally, the optimal dwell times at coverage locations, charging time, and agent trajectories are determined so as to maximize coverage over a given time interval. In particular, our controller guarantees that at any time there is at most one agent leaving the team for energy repletion.    
\end{abstract}

\section{INTRODUCTION}
Systems consisting of cooperating mobile agents are often used to perform
tasks such as coverage \cite{zhong2011distributed, leonard2013nonuniform, kantaros2015distributed, meng2018hybrid},
surveillance \cite{tang2005motion}, monitoring and sweeping
\cite{smith2012persistent}. A coverage task is one where agents are deployed
so as to cooperatively maximize the coverage of a given mission space
\cite{meguerdichian2001coverage}, where \textquotedblleft
coverage\textquotedblright\ is usually measured through
the joint detection probability of random events \cite{hossain2012impact}. Widely used methods to solve the coverage problem include distributed
gradient-based \cite{zhong2011distributed} and
Voronoi-partition-based algorithms \cite{cortes2004coverage}. These
approaches typically result in locally optimal solutions, hence possibly poor
performance. To escape such local optima, a boosting function approach is
proposed in \cite{sun2014escaping} where the performance is ensured to be
improved. Recently, the coverage problem was also approached by exploring the
submodularity property \cite{zhang2016string} of the objective function, and
a greedy-gradient algorithm is used to guarantee a provable bound relative to the
optimal performance \cite{sun2019exploiting}.

In most existing coverage problem settings, agents are assumed to have unlimited on-board
energy to perform the coverage task. However, in practice, battery-powered
agents can only work for a limited time in the field \cite{leahy2016persistent}. For example, most
commercial drones powered by a single battery can fly for only about 15
minutes. Developing distributed algorithms for multi-agent systems with energy constraints is considered in \cite{tokekar2011energy, jaleel2013probabilistic, aksaray2016dynamic, setter2017energy}. A consensus algorithm is proposed in~\cite{setter2017energy} to make multiple robots with energy constraints quickly reach a rendezvous point.  Unlike other multi-agent
energy-aware algorithms in the aforementioned references whose purpose is to reduce energy cost, we assume that a charging station is available for agents to replenish their energy according to some policy. We take into account such energy constraints
and add another dimension to the traditional coverage problem. The basic setup
is similar to that in \cite{zhong2011distributed}. Agents interact with the
mission space through their sensing capabilities which are normally dependent
upon their physical distance from an event location. Outside its sensing
range, an agent has no ability to detect events. 
The objective is to maximize an overall environment coverage measure by
controlling the movement of all agents in a {\it centralized} manner while
guaranteeing that no agent runs out of energy while in the mission space.

A {\it decentralized} feasible solution to this problem is proposed in~\cite{meng2018multi} via a hybrid system approach. Due to the decentralized nature of the algorithm in~\cite{meng2018multi}, agents have limited local information. Therefore, the performance is also degraded by the information inaccessibility.  This raises the question of what would be the ``best'' performance when all information is available, which motivates us to study the coverage problem via a centralized approach. Therefore, we revisit the same problem formulation as in~\cite{meng2018multi}. The objectives are to find the optimal centralized solution for multi-agent coverage problems and to characterize the ``price of decentralization''. To this end, we assume that the environment to be monitored is completely known. Then, the optimal coverage (OCV) locations of the agents while none of them needs recharging can be found through distributed gradient-based algorithms \cite{zhong2011distributed}, or improved versions such as the greedy-gradient based algorithm in~\cite{sun2019exploiting}, especially when obstacles are present. When an agent needs recharging, it will head to the charging station. If the agent still performs the coverage task at the charging station, the OCV locations for the remaining agents can be found using the aforementioned approaches. The optimal locations for all agents in this case are referred to as ``{\it optimal charging (OCH) formation}''. Therefore, every agent's behavior is to switch between the OCV formation and the OCH formation. The missing piece for the overall optimality is to determine the optimal way to manage the transient behavior between these two modes. However, this turns out to be a challenging task. To find a near-optimal solution for the transient between switches, a Traveling Salesman Problem (TSP) is solved to find the shortest total distances if an agent traverses all locations in both the OCV and OCH formations. The solution from the TSP dictates the order of locations being visited by any agent. Next, when the switching times of all agents are synchronized, the objective becomes minimizing the transient time and the energy cost during that time. By ``synchronization'', we mean that all agents leave the OCV formation at the same time, and arrive at the OCH formation at the same time. Therefore, the transient time is determined by the agent which travels the longest distance. The speeds of other agents can be determined by the transient time and the travel distance. 

 The main contributions of this paper are as follows: (i). Model the collective behavior of agents by formations (OCV formations and OCH formations) and transitions between them. (ii). Find the optimal orders of agents visiting different optimal locations including the charging station through solving a TSP. (iii) Derive optimal speed profiles for all agents during the transient time to minimize the transition cost. (iv). Quantify the price of decentralization through simulation experiments. 

\section{Problem Formulation}
Consider a bounded mission space $\mathcal{S}\in\mathbb{R}^{2}$. The value of a point $(x,y)\in \mathcal{S}$ in the mission space is characterized by a reward function $R(x,y)$, where $R(x,y)\geq 0$ and $\int\int_{\mathcal{S}}R(x,y)dxdy <\infty$.
The value of $R(x,y)$ is monotonically increasing in the importance associated with point $(x,y)$. 
If all points in $\mathcal{S}$ are
treated indistinguishably, $R\left(  x,y\right)=\sigma$ for any $(x,y) \in \mathcal{S}$, where $\sigma>0$ is a constant.
A team of mobile agents labeled by $\mathcal{V}=\{1,2,\ldots,N\}$ is deployed in the mission space to detect possible events that occur in it.
Each agent has an isotropic sensing system with range $\delta_{i}$, that is,
an agent located at $(x_{i},y_{i})$ is able to cover the area
\[
\Omega_{i}\left(  x_{i},y_{i}\right)  =\left\{  \left(  x,y\right)  |\left(
x-x_{i}\right)  ^{2}+\left(  y-y_{i}\right)  ^{2}\leq\delta_{i}^{2}\right\}
\text{.}%
\]
The sensing probability of an agent for a point $\left(  x,y\right)$ within
its sensing range $\Omega_{i}\left(  x_{i},y_{i}\right)  $ is characterized by
the sensing function $p_{i}\left(  x,y,x_{i},y_{i}\right)  \in\left[
0,1\right]  $, and it depends on the distance between the agent location $\left(
x_{i},y_{i}\right)  $ and the point $\left(  x,y\right)  $. In particular, it
is monotonically decreasing in the distance between $\left(  x_{i}%
,y_{i}\right)  $ and $\left(  x,y\right)  $ and if a point $\left(
x,y\right)  $ is out of the sensing range of agent $i$, that is, $\left(
x,y\right)  \notin\Omega_{i}\left(  x_{i},y_{i}\right)  $, then $p_{i}\left(
x,y,x_{i},y_{i}\right)  =0$. For any given point $\left(  x,y\right)  $ in the
sensing range of multiple agents, assuming independence among agent sensing
capabilities, the joint event detection probability is given by \cite{zhong2011distributed}%
\begin{equation}
P\left(  x,y,\mathbf{s}\right)  =1-\prod\nolimits_{i=1}^{N}\left[
1-p_{i}\left(  x,y,x_{i},y_{i}\right)  \right]  \text{.} \label{sq}%
\end{equation}
Figure~\ref{figsrf} depicts the event detection probability of a single agent (Fig.~\ref{fig:gull}) and two agents with overlapping sensing range (Fig.~\ref{fig:tiger}), where $p(x,y,x_{i},y_{i})=1-\frac{(x-x_{i})^2+(y-y_{i})^2}{\delta_{i}^2}$ \cite{meng2018hybrid}.
\begin{figure}[hpbt]
\begin{center}
\begin{subfigure}[b]{0.45\textwidth}
\includegraphics[width=\textwidth]{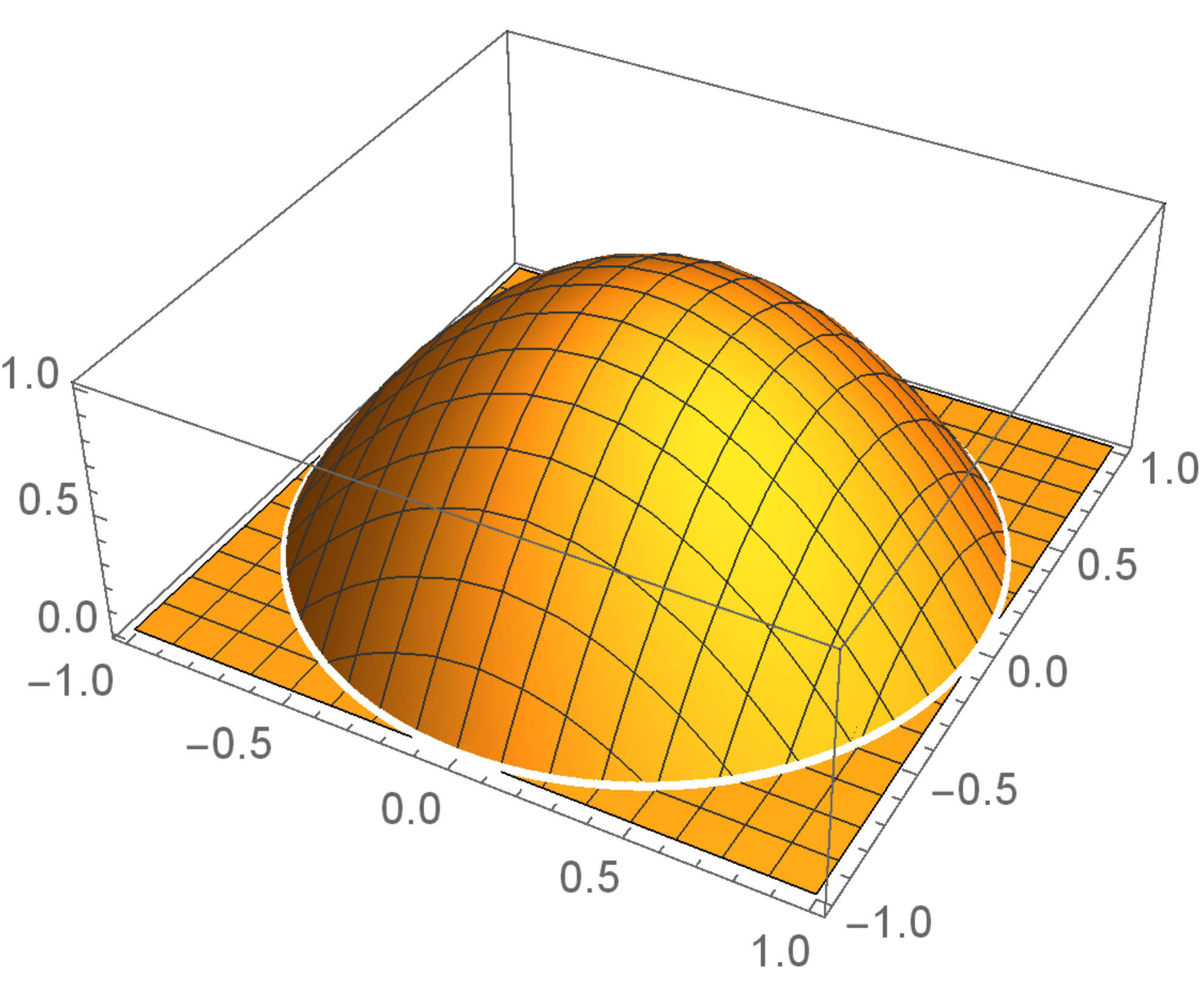}
\vspace{-0.9cm}
\caption{A single agent at (0,0)}
\label{fig:gull}
\end{subfigure}
\begin{subfigure}[b]{0.45\textwidth}
\vspace{1cm}
\includegraphics[width=\textwidth]{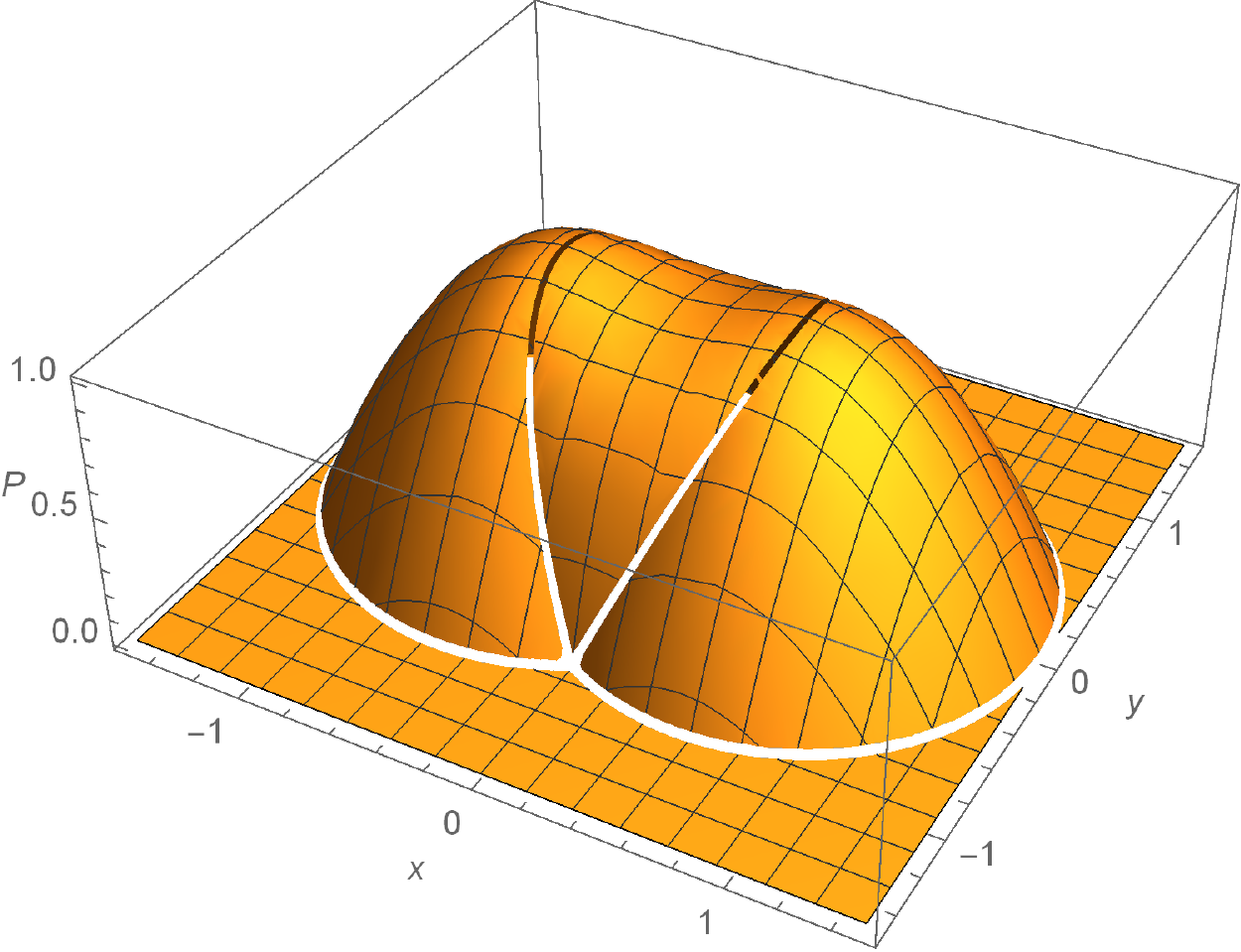}
\caption{Two agents at (0.5,0) and (-0.5, 0)}
\label{fig:tiger}
\end{subfigure}
\end{center}
\caption{Sensing probability of an area with one and two agents performing coverage.}%
\label{figsrf}%
\end{figure}

Finally, the coverage performance of the mobile agent team to the area $\mathcal{S}$ is defined as%
\begin{equation}
H\left(  \mathbf{s}\right)  =\int\int_{\mathcal{S}}R\left(  x,y\right)  P\left(
x,y,\mathbf{s}\right)  dxdy, \label{op}
\end{equation}
where $\mathbf{s}=[s_{1}^{T},\ldots,s_{N}^{T}]^{T}$ with $s_{i}=[x_{i},y_{i}]^{T}$ is a column vector that
contains all agent positions.
Note that $H\left(
\mathbf{s}\right)  $ is a function mapping a vector $\mathbf{s\in%
\mathbb{R}
}^{2N}$ into $%
\mathbb{R}
$.

To find the optimal locations of all agents is a static optimization problem, which has been extensively studied \cite{cortes2004coverage, zhong2011distributed, sun2019exploiting}. Here we are interested in a {\it dynamic} coverage control problem with energy constraints, where each agent is associated with two state variables: location variable $s_{i}(t)$ and state-of-charge (SOC) variable $0\leq q_{i}(t)\leq 1$, which is a percentage of the battery level. The agents' sensing, motion, and communication activities are all powered by batteries, and there is a charging station available at $(0,0)$ for all agents to replenish their energy. We assume that there is only one outlet in the charging station. In other words, only one agent can be charged at any time.
The agent's motion is described by the following kinematic equations:%
\begin{equation}
\dot{x}_{i}(t)=v_{i}(t)\cos[\theta_{i}(t)], \quad \dot{y}_{i}(t)=v_{i}(t)\sin[\theta_{i}(t)] \label{md}
\end{equation}
where $v_{i}(t)$  and $\theta_{i}(t)$ denote the instantaneous speed and heading of agent $i$ at time $t$, respectively. We assume that $v_{i}\left(  t\right)
\in\left[  0,\bar{v}\right]$, where $\bar{v}$ is the maximum speed of an agent. For simplicity, assume that the speed and angular state can be controlled directly.

The state-of-charge (SOC) state satisfies the following dynamic equation:%
\begin{equation}
\dot{q}_{i}(t)=I_i(t)f(q_{i}(t),b_{i}(t))+(1-I_i(t))g(q_{i}(t),v_{i}(t),b_{i}(t)) \label{soc}
\end{equation}%
where $I_{i}(t)=1$ means the agent is in charging mode and $f(q_{i}(t),b_{i}(t))\geq 0$, and $I_{i}(t)=0$ means the agent is in energy depletion mode and $g(q_{i}(t),v_{i}(t),b_{i}(t))\leq 0$. Moreover, $g(q_{i}(t),v_{i}(t),b_{i}(t))=0$ when $v_i(t)=0$ and $b_{i}(t)=0$. The control $b_{i}(t) \in \{0,1\}$ is a binary variable to indicate ``on'' ($b_{i}(t)=1$) or ``off'' $b_{i}(t)=0$ of the sensing functionality of an agent. In other words, an agent is in energy conservation mode if there is neither motion nor sensing.


Our objective is to maximize the coverage of the mission space $\mathcal{S}
\in\mathbb{R}^{2}$ over a time interval $[0,T]$, and at the same time to keep all
agents alive, that is, $q_{i}\left(  t\right)  >0$ for all $t\in\lbrack0,T]$.
The case $q_{i}\left(  t\right)  =0$ can occur only at the charging station
$\left(  0,0\right)  $. Therefore, we consider the following optimization
problem for each agent $i$:%
\begin{align}%
&\underset{v\left(  t\right),\text{ } \theta\left(  t\right),\text{ } b(t) }{\max}
\frac{1}{T}\int_{0}^{T}H\left(  \mathbf{s}\left(  t\right)  \right)  dt \label{covob}\\
\text{s.t. \quad} &(\ref{md}) \text{ and } (\ref{soc})\\ 
& I_{i}(t)=1 \text{ when } s_{i}(t)=0, \label{cons2}\\
&q_{i}\left(  t\right)  >0 \text{ when } s_{i}\left(  t\right) \neq0, \label{cons3}\\
&0\leq v_{i}(t)\leq \bar{v}, \, 0\leq \theta_{i}(t)< 2\pi, \,  0\leq q_{i}(t)\leq 1\\
& b_{i}(t)\in \{0,1\}, I_{i}(t) \in \{0,1\},  i=1,\ldots,N,\\
&0 \leq \sum\nolimits_{i=1}^{N} I_{i}(t) \leq 1, \label{cons4}
\end{align}%
where $T$ is a given time horizon, $v(t)=[v_{1}(t),\ldots,v_{N}(t)]^{T}$, $\theta(t)=[\theta_{1}(t),\ldots,\theta_{N}(t)]^{T}$, $b(t)=[b_{1}(t),\ldots,b_{N}(t)]^{T}$, and the coverage metric $H\left(  \mathbf{s}\left(  t\right)  \right)$ is defined in (\ref{op}). In this paper, we consider $b_i(t) =1$ for any $t\geq 0$, that is, an agent always senses the environment to perform the coverage task. The constraints (\ref{cons2}) indicate that an agent is in charging mode whenever it arrives at the charging station; (\ref{cons3}) prevents agents from dying, i.e., running out of energy, in the mission space;  (\ref{cons4}) ensures that only one agent can be
served at the charging station at any time.
\begin{remark}The same problem was solved in~\cite{meng2018multi} by a decentralized approach. Here we try to revisit the problem by using a centralized approach. In the centralized approach, all information is available, and every agent can be controlled in a centralized way. In the decentralized approach, all agents cooperatively find the OCV and OCH locations using only local information. When multiple agents compete for the charging station, the charging station works as a controller to schedule the charging of all competing agents. The purpose of this paper is to compare the performance of the two different approaches.
\end{remark}

\section{Main Results}
Previous work in~\cite{meng2018multi} solves this problem from an individual agent point of view, where the behavior of an agent is modeled through three different modes: coverage mode, to-charge mode and in-charge mode. In this paper, however, we aim to solve this problem from the team point of view. We would ultimately like to maximize the coverage level in (\ref{covob}) and minimize transient times that occur between the OCV and OCH formations. This comes down to solving the following problems:
\begin{enumerate}
\item In Section~\ref{ol}, find the OCV and OCH locations for all agents.
\item In Section~\ref{sp}, solve a TSP to get an optimal path connecting all OCV and OCH locations found in Section~\ref{ol}.
\item In Section~\ref{fea}, establish problem feasibility.
\item In Section~\ref{opvel}, solve for the optimal speed problem over transient intervals assumed in Section~\ref{fea}. 
\item In Section~\ref{odtct}, maximize the coverage performance by optimizing dwell and charge times based on the feasibility condition found in Section~\ref{fea} and the optimal speed profile found in Section~\ref{opvel}.
\end{enumerate}

\subsection{Optimal Locations}\label{ol}
Let us assume that the environment is known, that is, $R(x,y)$ is known. If all agents are in the ``coverage mode'', they should be in the OCV locations as determined by standard gradient algorithms as in~\cite{zhong2011distributed}.

Let us denote the OCV locations for $N$ agents as ${\bf s}^{1}=\{s_{1}^{1},\ldots,s_{N}^{1}\}$ for the mission space $\mathcal{S}$. 
By assuming that the agent also performs the coverage task while resting at the charging station, we can calculate the OCH locations of the remaining $N-1$ agents by constraining one agent to be at $(0,0)$. Therefore, the OCH locations can be found using the gradient method proposed in~\cite{zhong2011distributed}. Let ${\bf s}^{2}=\{s_{1}^{2}, \ldots,s_{N-1}^{2},s_{N}^{2}\}$ be the OCH locations with $s_{N}^{2}=(0,0)$. Therefore, when all agents have enough energy, the optimal choice is to occupy all locations at ${\bf s}^{1}$. When an agent is at the charging station, the optimal choice for all agents is at the locations specified by ${\bf s}^{2}$. Whenever an agent leaves or re-joins the team, the agents switch between ${\bf s}^{1}$ and ${\bf s}^{2}$. Assume that all agents are of the same type and have the same initial SOC. The optimal scheduling is to let agents take turns to visit the charging station. Therefore, we are essentially transforming the original problem into a Multi-Agent TSP (MATSP) in the next section.

\subsection{Shortest Path}\label{sp}
Before proceeding further, let us give the following standard definitions which can be found in~\cite{diestel2010graph} to model the relationship between locations in ${\bf s}^{1}$, and ${\bf s}^{2}$.
\begin{definition}A graph $\mathcal{G}$ is called {\it bipartite} if its vertex set can be partitioned into two parts $\mathcal{V}_{1}$ and $\mathcal{V}_{2}$ such that every edge has one end in $\mathcal{V}_{1}$
and one in $\mathcal{V}_{2}$.
\end{definition}

\begin{definition}A bipartite graph in which every two vertices from different partition parts are adjacent is called {\it complete}.
\end{definition}

If $|\mathcal{V}_{1}|=|\mathcal{V}_{2}|=r$, we abbreviate the complete bipartite graph to $\mathcal{K}_{r}^{2}$, in which every part contains exactly $r$ vertices.

When agent 1 (without loss of generality) switches to ``to-charge'' mode, the locations $\{s_{2}^{1}, \ldots,s_{N}^{1}\}$ are not optimal for the remaining $N-1$ agents.
Therefore,  the remaining agents which are in the ``coverage mode'' need to switch to the OCH locations. When this process repeats, it turns out that an agent will visit all optimal locations in ${\bf s}^{1}$ and ${\bf s}^{2}$. Therefore, this process boils down to finding the shortest path for an agent to visit all optimal locations and return to its location. This is exactly the MATSP with certain constraints, i.e., when an agent is in
one of the OCV locations, it has to switch to one of the OCH locations. Thus, we can use the bipartite graph to model such constraints. As the agents switch
between the formations, we need to minimize the total traveled distance during transient times.
Let $\mathcal{K}_{N}^{2}=(\mathcal{V},\mathcal{E})$ denote the underlying topology, where the vertex set $\mathcal{V}$ can be partitioned into two sets: $\mathcal{V}_{1}={\bf s}^{1}$ and $\mathcal{V}_{2}={\bf s}^{2}$ such that $\mathcal{V}_{1}\cup \mathcal{V}_{2}=\mathcal{V}$, $\mathcal{V}_{1}\cap \mathcal{V}_{2} =\emptyset$ and $|\mathcal{V}_{1}|=|\mathcal{V}_{2}|=N$. Every edge in $\mathcal{V}$ has one end in $\mathcal{V}_{1}$ and the other end in $\mathcal{V}_{2}$ and vertices in the same set are not adjacent. In addition, $\mathcal{K}_{N}^{2}$ is complete, that is, every two vertices from different sets are adjacent. The weight of every edge is the distance between the two vertices.

Finding the shortest transient distance is equivalent to finding the shortest path in the graph $\mathcal{K}_{N}^{2}$. This is a MATSP, which can be solved by integer linear programming.

The underlying assumption is that when an agent switches to ``to-charge'' mode, no other agents will switch to the same mode until the agent returns and the OCV formation is attained. We will find a condition to guarantee that this assumption holds at all times.

\subsection{Feasibility}\label{fea}
Feasibility in this case means that the number of agents in both ``to-charge'' and ``in-charge'' modes are less than two. For simplicity, let us assume that the behavior of all agents is synchronized, that is, they start and finish the process of switching from $\mathcal{V}_{1}$ to $\mathcal{V}_{2}$ at the same time, and vice versa (i.e., from $\mathcal{V}_{2}$ to $\mathcal{V}_{1}$).
The intuition behind this assumption is that the coverage performance depends on the agents' relative distances.

Then, the problem reduces to finding four critical times: (1) the charging time $\tau_{c}$ at the charging station,
(2) the dwell time $\tau_{d}$ of agents on the OCV locations, (3) the transient time $\tau_{t}^{N-1}$ from the OCH locations to the OCV locations, and (4) the transient time $\tau_{t}^{N}$ from the OCH locations to the OCV locations. Note that the dwell time of agents on
the OCH locations is exactly equal to the charging time at the charging station.


Without loss of generality, we can assume that the optimal path to visit the locations for any agent follows the order: $1\rightarrow 2\rightarrow 3 \rightarrow \cdots 2N-1 \rightarrow 2N$, where the nodes with odd numbers
belong to the OCV locations, the nodes with even numbers belong to the OCH locations, and $2N$ is the charging station. Let us define $q_{i}^{-}$ and $q_{i}^{+}$ as the energy when agents arrive at node $i$ and leave node $i$, respectively, and $d_{2i-1}^{2i}$ as the distance between node
$2i-1$ and $2i$. Let us proceed backwards starting at node $2N$. Clearly, we must have $q_{2N}^{-}\geq 0$ to make the problem feasible. Therefore,
\[
q_{2N}^{-}=q_{2N-1}^{+}+h(q_{2N-1}^{+}, \tau_{t}^{N}, d_{2N-1}^{2N})\geq 0
\]
where $h(\cdot)$ is an energy cost function determined by (\ref{soc}) when $I_i(t)=0$ under the assumption of the optimal speed (which will be determined in Section~\ref{opvel}). If the process is repeated recursively, the minimum energy at node $2N-1$ will be%
\begin{align*}
q_{2N-1}^{+}&=q_{2N-1}^{-}+h(q_{2N-1}^{-}, \tau_{d}, 0)\\
q_{2N-1}^{-}&=q_{2N-2}^{+}+h(q_{2N-2}^{+}, \tau_{t}^{N-1}, d_{2N-2}^{2N-1})
\end{align*}
In general, the minimum energy requirements for the locations $2i$ and $2i-1$ are%
\begin{align}
q_{2i}^{+}&=q_{2i}^{-}+h(q_{2i}^{-}, \tau_{c}, 0) \notag\\
q_{2i-1}^{-}&=q_{2i-1}^{+}+h(q_{2i-1}^{+}, \tau_{t}^{N}, d_{2i-1}^{2i}) \label{feas_c1}
\end{align}%
and%
\begin{align}
q_{2i-1}^{+}&=q_{2i-1}^{-}+h(q_{2i-1}^{-}, \tau_{d}, 0) \notag\\
q_{2i-2}^{-}&=q_{2i-2}^{+}+h(q_{2i-2}^{+}, \tau_{t}^{N-1}, d_{2i-2}^{2i-1}), \label{feas_c2}
\end{align}%
respectively.
Eventually, the minimum energy for node 1 will be%
\[
q_{1}^{+}=q_{1}^{-}+h(q_{1}^{-}, \tau_{d}, 0).
\]
Also note that%
\begin{align}
q_{1}^{-}&=q_{2N}^{+}+h(q_{2N}^{+},\tau_{t}^{N-1}, d_{2N}^{1}) \label{feas_c10}\\
q_{2N}^{+}&=q_{2N}^{-}+\kappa(q_{2N}^{-},\tau_{c}) \label{feas_c11}
\end{align}%
where $\kappa(\cdot)$ is the solution of the differential equation (\ref{soc}) with the initial condition $q_{2N}^{-}$ and $I_{i}(t)=1$. 

We now include an iteration index $k=1,2,\dots$, and write%
\[
q_{2N}^{-}(k)\geq q_{2N}^{-}(k-1)\geq 0
\]
We then need to solve the following optimization problem%
\begin{align}
\text{Feasibility Problem} &\min_{\tau_c, \tau_d=0, \tau_{t}^{N},\tau_{t}^{N-1}} q_{2N}^{-}(k-1) \label{feainq3}\\
\text{subject to}
\qquad & (\ref{feas_c1}) \text{ and } (\ref{feas_c2}) \text{ for } i=1,\dots,N \\
\qquad &q_{2N}^{-}(k)\geq q_{2N}^{-}(k-1)\geq0 \label{feainq}
\end{align}%
for any $k\geq 1$, where we set $\tau_d=0$ to capture the extreme case that the dwell time at the OCV locations is zero for all agents. Note that $q_{2N}^{-}(k)$ can be expressed as a function of $q_{2N}^{-}(k-1)$, $\tau_c$, $\tau_d$, $\tau_{t}^{N}$ and $\tau_{t}^{N-1}$.

Only if a solution to (\ref{feainq3})-(\ref{feainq}) exists we can further maximize the dwell time $\tau_d$. Therefore, it is clear that (\ref{feainq3})-(\ref{feainq}) defines the feasibility problem. We want to find the control variables $\tau_c$, $\tau_{t}^{N}$ and $\tau_{t}^{N-1}$ so that the SOC does not decrease during a cycle.
This condition determines the feasibility of the problem. Once the minimum $q_{2N}^{-}(k-1)$ is obtained, we can calculate $q_{2N}^{+}(k-1)$ using (\ref{feas_c11}), and then $q_1^{-1}(k)$ using (\ref{feas_c10}) to start a new iteration. Repeating the calculation forward, we are able to compute $q_{i}^{-}(k)$, and $q_{i}^{+}(k)$ using (\ref{feas_c2}) and (\ref{feas_c1}) for $i=1,\dots,2N$.

\subsection{Optimal Speed}\label{opvel}
In the previous section, we assume that the energy cost is calculated under the optimal speed of an agent during a transient period in which a switch between OCV and OCH formations takes place. Here we will derive this optimal speed when the travel time and distance of a transient segment of an agent trajectory are given.
During the transient period $\tau$, the optimal speed can be determined so as to minimize the energy cost.
Therefore, the following optimization problem is formulated:
\begin{align}%
&\underset{v_{i}\left(  t\right),\text{ } \theta_{i}\left(  t\right)}{\min}
\int_{t_{0}}^{t_{0}+\tau} \dot{q}_{i}(t)dt\\
\text{subject to \quad} &  (\ref{md})  \text{ and } (\ref{soc}) \\
&0\leq v_{i}(t)\leq \bar{v}\\
& s_{i}\left(  t_{0}\right)  = \underline{s}\\
& s_{i}\left(  t_{0}+\tau\right)  = \bar{s}\\
&q_{i}\left(  t_{0}\right)  = \underline{q},
\end{align}%
where $\underline{s}$ and $\bar{s}$ are initial and final positions of agent $i$,  respectively, and $ \underline{q}$ is the initial SOC of agent $i$.
\begin{theorem} Assume that  the energy model in (\ref{soc}) when $I_{i}(t)=0$ has the following linear form%
\[
g(q_{i}(t),v_{i}(t),1)=-\alpha v_i(t)-\beta
\]%
where $\alpha>0$ and $\beta>0$ are two constants. Then, the optimal solutions to the above optimization problem are %
\[
v^{\ast}(t) = \frac{\|\bar{s}-\underline{s}\|}{\tau}
\]%
and%
\[
\theta^{\ast}(t)= \phase{\bar{s}-\underline{s}}
\]%
for $t \in [t_{0},t_{0}+\tau)$, where $\phase{\bar{s}-\underline{s}}$ is the heading from $\underline{s}$ to $\bar{s}$.
The minimum energy cost is%
\[
\alpha \|\bar{s}-\underline{s}\| + \beta \tau.
\]
\end{theorem}
\begin{proof}The Hamiltonian function is defined as%
\begin{align*}
\mathcal{H}(s_{i},v_{i},\theta_{i},q_{i},t)=\,&\dot{q}_{i}+\lambda_{x}v_{i}\cos(\theta_{i})+\lambda_{y}v_{i}\sin(\theta_{i})\\
&+\lambda_{q}g(q_{i},v_{i},1).
\end{align*}
We have the co-state equations:
\[
-\dot{\lambda}_{x}=\frac{\partial \mathcal{H}}{\partial x_{i}}=0, \quad -\dot{\lambda}_{y}=\frac{\partial \mathcal{H}}{\partial y_{i}}=0.
\]
Therefore, we know that $\lambda_{x}$ and $\lambda_{y}$ are two constants. From the stationarity condition, we have%
\[
\frac{\partial \mathcal{H}}{\partial \theta_{i}}=-\lambda_{x}v_{i}\sin(\theta_{i})+\lambda_{y}v_{i}\cos(\theta_{i})=0.
\]
Then, we know that $\theta_{i}$ is also a constant determined by the initial and final positions. Let $\lambda_{x}=\lambda_{\theta} \cos(\theta_{i})$ and $\lambda_{y}=\lambda_{\theta} \sin(\theta_{i})$ with a constant $\lambda_{\theta}$. Thus, the Hamiltonian function becomes
\begin{align*}
\mathcal{\mathcal{H}}(s_{i},v_{i},\theta_{i},q_{i},t)=\,&\dot{q}_{i}+\lambda_{\theta} v_{i}\cos^2(\theta_{i})+\lambda_{\theta} v_{i}\sin^2(\theta_{i})\\
&+\lambda_{q}g(q_{i},v_{i},1)\\
   =\,&(1+\lambda_{q})g(q_{i},v_{i},1)+\lambda_{\theta} v_{i}.
\end{align*}
Then, we have%
\begin{align*}
-\dot{\lambda}_{q}=\frac{\partial \mathcal{H}}{\partial q_{i}}=(1+\lambda_{q})\frac{\partial g(q_{i},v_{i})}{\partial q_{i}}.
\end{align*}
Based on the linear form of $g(q_{i},v_{i},1)$, we have%
\[\frac{\partial g(q_{i},v_{i},1)}{\partial q_{i}}=0.\]%
 We know that $\lambda_{q}$ is a constant and%
\[
\mathcal{H}(s_{i},v_{i},\theta_{i},q_{i},t)=[-(1+\lambda_{q})\alpha +\lambda_{\theta} ]v_{i}(t)-(1+\lambda_{q})\beta .
\]%
Since $\mathcal{H}$ is not an explicit function of time $t$, we have $\dot{\mathcal{H}}=0$. Thus, we obtain $\dot{v}_i(t)=0$. Therefore, $v_{i}$ is a constant determined by the distance between the initial and final positions and the travel time $\tau$.
\end{proof}

The energy cost is determined by both the distance and the travel time $\tau$. Therefore, to reduce the transient time, it is always optimal for agents who travel the longest distance during the transient times to use the maximum speed when the energy consumption model is a linear function of the speed.

\subsection{Optimal Dwell Time and Charging Time}\label{odtct}
Once the solution of the TSP is available, the remaining task is to maximize the coverage time and minimize the transient time during a cycle. Then, we define a duty cycle-like objective function below as the fraction of the total cycle $\tau_c+\tau_d+\tau_{t}^N+\tau_{t}^{N-1}$ used by the dwell time $\tau_d$ (which provides maximum coverage):
\begin{align}
&\max_{\tau_c, \tau_d} \frac{\tau_d}{\tau_c+\tau_d+\tau_{t}^N+\tau_{t}^{N-1}} \label{opcd}\\
\text{subject to}
\qquad & \tau_{c}\leq \bar{\tau}_c\label{opcd1}\\
\qquad & \{\tau_c, \tau_d \}\in \mathcal{F} \label{opcd2}
\end{align}%
where $\mathcal{F}$ is the set of all pairs of $(\tau_c, \tau_d )$ which can satisfy the inequality (\ref{feainq}) in Section~\ref{fea}, and $\bar{\tau}_c$ is the time when the battery is fully charged starting with an initial SOC $q_{2N}^{-}$. In addition:%
\begin{equation}
\tau_{t}^N = \frac{\bar{d}^N}{\bar{v}}, \qquad \tau_{t}^{N-1} = \frac{\bar{d}^{N-1}}{\bar{v}},
\end{equation}
where $\bar{d}^N = \max_{i=1,\dots,N}{d_{2i-1}^{2i}}$, and $\bar{d}^{N-1} = \max_{i=1,\dots,N}{d_{2i}^{2i-1}}$.

The first constraint requires an agent to leave the charging station once its battery is fully charged. This is motivated by the fact shown in our previous work in~\cite{meng2018hybrid} that it is optimal to fully charge an agent. The second constraint ensures that the charging time and the dwell time must satisfy the feasibility constraint. 

\section{SIMULATION EXAMPLES}
Let us consider a small network with 3 agents to cover a $600 \times 500$ rectangular mission space. By using the gradient approach \cite{zhong2011distributed}, the OCV locations of all three agents with a sensing range $220$ are found to be $s_{1}^{1}=(186.7, 119.3)$, $s_{2}^{1}=(160.3, 371.1)$,
 and $s_{3}^{1}=(451.4, 290.4)$ shown in blue in Fig.~\ref{fig:my-label}, and the OCV locations are $s_{1}^{2}=(0,0)$, $s_{2}^{2}=(169.3, 320.2)$ and $s_{3}^{2}=(430.6, 185.0)$ shown in red in Fig.~\ref{fig:my-label}. The charging station is located at $s_{1}^{2}$.
 Let us assume that the charging dynamics in (\ref{soc}) have the form $f(q_{i}(t),1)=c-\beta$, and the energy depletion dynamics in (\ref{soc}) have the form $g(q_{i}(t),v_{i}(t),1)=-\alpha v_{i}(t)-\beta$, where $\alpha$, $\beta$ and $c$
are three constants. For a properly defined problem, the following constraint should be satisfied
\begin{equation}
 c\geq 3(a \bar{v}+\beta) \label{fea3},
\end{equation}%
where $\bar{v}$ is the maximum allowed speed of all agents. By treating the charging station as a server, the charging rate is $c$ if it is occupied at all times, and the worst case
energy depletion rate over three agents is $3(a \bar{v}+\beta)$. Thus, the condition (\ref{fea3}) ensures the feasibility to prevent any agent from running out of energy in the mission space.
 By solving the TSP, the shortest path is 
 $s_{2}^{1}\rightarrow s_{2}^{2}\rightarrow s_{3}^{1}\rightarrow s_{3}^{2}\rightarrow s_{1}^{1}\rightarrow s_{1}^{2}\rightarrow s_{2}^{1}$. The total traveling distance is $2388$.
 \begin{figure}[!htb]
\center{\includegraphics[trim=0 0cm 0 0cm, width=0.48\textwidth]{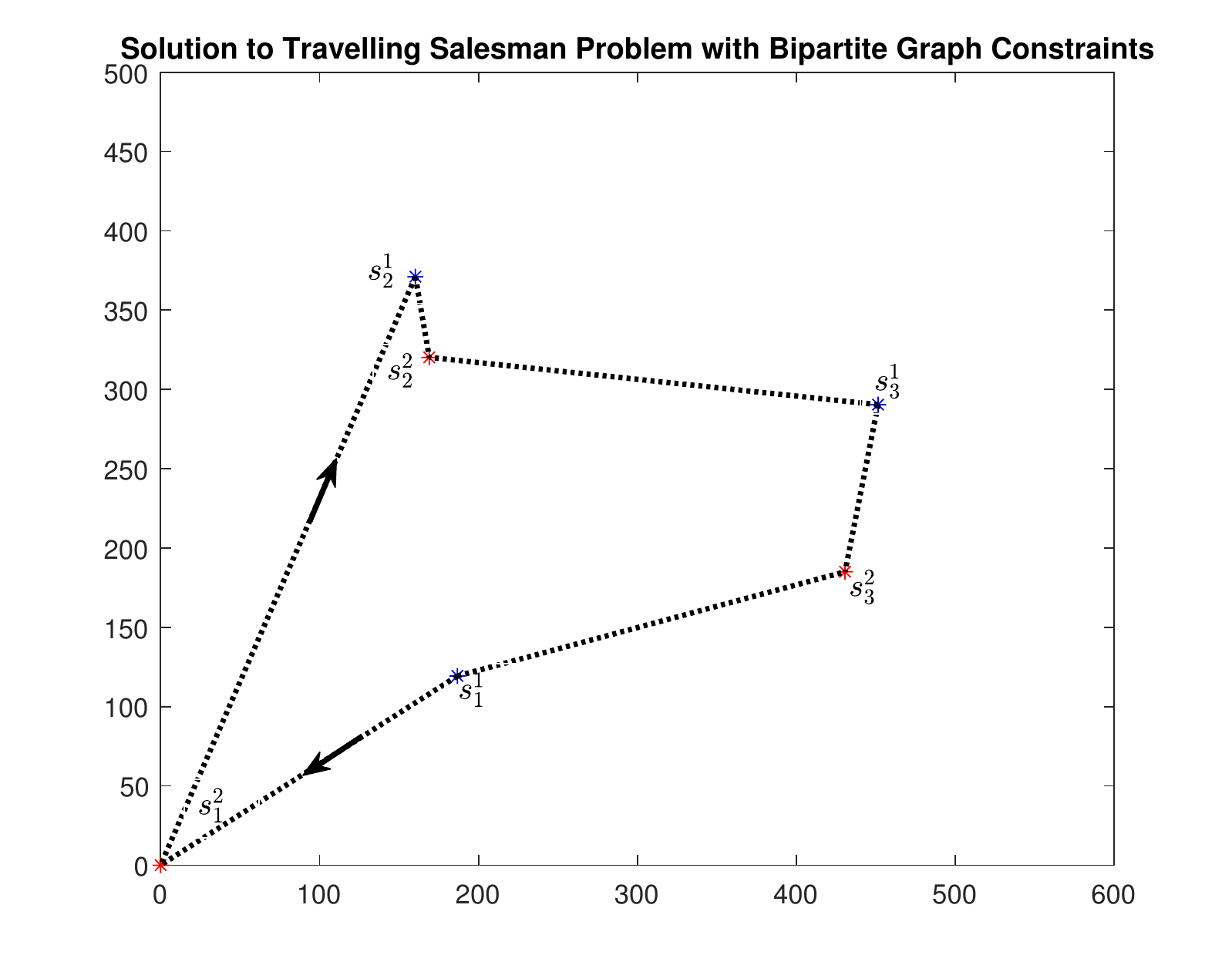}}
\caption{The shortest path for the TSP}
\label{fig:my-label}
 \end{figure}
 
Let us solve the feasibility problem (\ref{feainq3})-(\ref{feainq}) first. Note that node $6$ is defined as the charging station in Section~\ref{fea}. Assume that $q_{6}^{-}=q_{0}$, that is, the SOC when an agent arrives at the charging station. The distances are $d_{s_{1}^{1}}^{s_{1}^{2}}=221.5612, d_{s_{1}^{2}}^{s_{2}^{1}}=252.5939, 
 d_{s_{2}^{1}}^{s_{2}^{2}}=107.4328, d_{s_{2}^{2}}^{s_{3}^{1}}=283.9381, 
d_{s_{3}^{1}}^{s_{3}^{2}}=51.6432, d_{s_{3}^{2}}^{s_{1}^{1}}=404.2416$. When the energy depletion model is linear in $v_{i}$, it is optimal to choose the shortest transient time. 
The lower bound of transient times $\tau_{t}^{2}$ and $\tau_{t}^{3}$ are determined by the distances and maximum speed. Therefore, we can choose%
\[
\tau_{t}^{3}=\frac{\max\{d_{s_{1}^{2}}^{s_{2}^{1}}, d_{s_{2}^{2}}^{s_{3}^{1}}, d_{s_{3}^{2}}^{s_{1}^{1}}\}}{\bar{v}}=\frac{404.2416}{\bar{v}}
\]
and
\[
\tau_{t}^{2}=\frac{\max\{d_{s_{1}^{1}}^{s_{1}^{2}}, d_{s_{2}^{1}}^{s_{2}^{2}}, d_{s_{3}^{1}}^{s_{3}^{2}}\}}{\bar{v}}=\frac{221.5612}{\bar{v}}
\]

After charging for $\tau_c$, 
 the SOC increases to $q_{0}+\tau_{c}(c-\beta)$. Then, the agent heads to $s_{2}^{1}$, and its SOC decreases to $q_{0}+\tau_{c}(c-\beta)-\alpha d_{4}^{2}-\beta\tau_{t}^{3}$, where the third term and the last term correspond to the energy cost of motion and sensing, respectively. 
To solve the feasibility problem (\ref{feainq3}), we set the dwell time at the OCV locations as zero. After one cycle, when an agent returns to the charging station, its SOC
becomes%
\[
q_{0}+\tau_{c}c-2388\alpha  - 3\beta\tau_{t}^{3}-3\beta\tau_{t}^{2}-3\beta \tau_c.
\]%
and we require:
\[
q_{0}+\tau_{c}c-2388 \alpha - 3\beta\tau_{t}^{3}-3\beta\tau_{t}^{2}-3\beta \tau_c\geq q_{0}.
\]

Therefore, in this case it is possible $q_{0}=0$, and the minimum charging time is
\[
\tau_c=\frac{2388 \alpha + 3\beta(\tau_{t}^{3}+\tau_{t}^{2})}{c-3\beta}
\]
Based on $q_0=0$, $\tau_c$, $\tau_{t}^{2}$, and $\tau_{t}^{3}$, we are able to calculate the minimum SOC for all 3 optimal locations as shown at the end of Section~\ref{fea}. 

If an agent stays at the charging station more than the minimum $\tau_c$, then the dwell time $\tau_d$ will not be zero. Therefore, we need to solve the optimization problem (\ref{opcd})-(\ref{opcd2}) to maximize $\tau_{d}$ and
its percentage during a cycle:
\begin{align*}
&\max_{\tau_c, \tau_d} \frac{\tau_d}{\tau_c+\tau_d+\tau_{t}^3+\tau_{t}^2}\\
\text{subject to}\\
&\tau_{c}\leq \frac{1}{c-\beta}\\
&\tau_c \geq \frac{2388 \alpha + 3\beta(\tau_{t}^{3}+\tau_{t}^{2}+\tau_d)}{c-3\beta}
\end{align*}
The first condition is to make sure that agents will not stay at the charging station when it is fully charged which corresponds to (\ref{opcd1}). The second condition is to guarantee that an agent will not run out of energy in the mission space, which corresponds to (\ref{opcd2}).

To solve the above optimization problem, the optimal solution occurs when the first inequality become equality. Then, we can write the relationship between $\tau_c$ and $\tau_d$ as $\tau_c=a+b\tau_d$. If we substitute $\tau_c$ by $a+b\tau_d$, we know that the larger $\tau_d$ leads to better performance. Therefore, the optimal solution for the above problem is to let the agent be fully charged, that is,%
\[\tau_c=\frac{1}{c-\beta},
\] and%
\[
\tau_d = \frac{1-2388 \alpha}{3\beta}-\frac{1}{c-\beta}-\tau_{t}^{3}-\tau_{t}^{2}
\]
Let us choose $\alpha = 0.0005$, $\beta=0.0005$, $c=0.01$, and $\bar{v}=50$.
The coverage performance of the above centralized algorithm is depicted in Fig.~\ref{f1_c}. The cycles are clearly visualized in the figure, where the top horizontal lines and the bottom horizontal lines correspond to the time when agents are in the OCV formation, and in the OCH formation, respectively.

\begin{figure}[htbp]
\includegraphics[width=\linewidth]{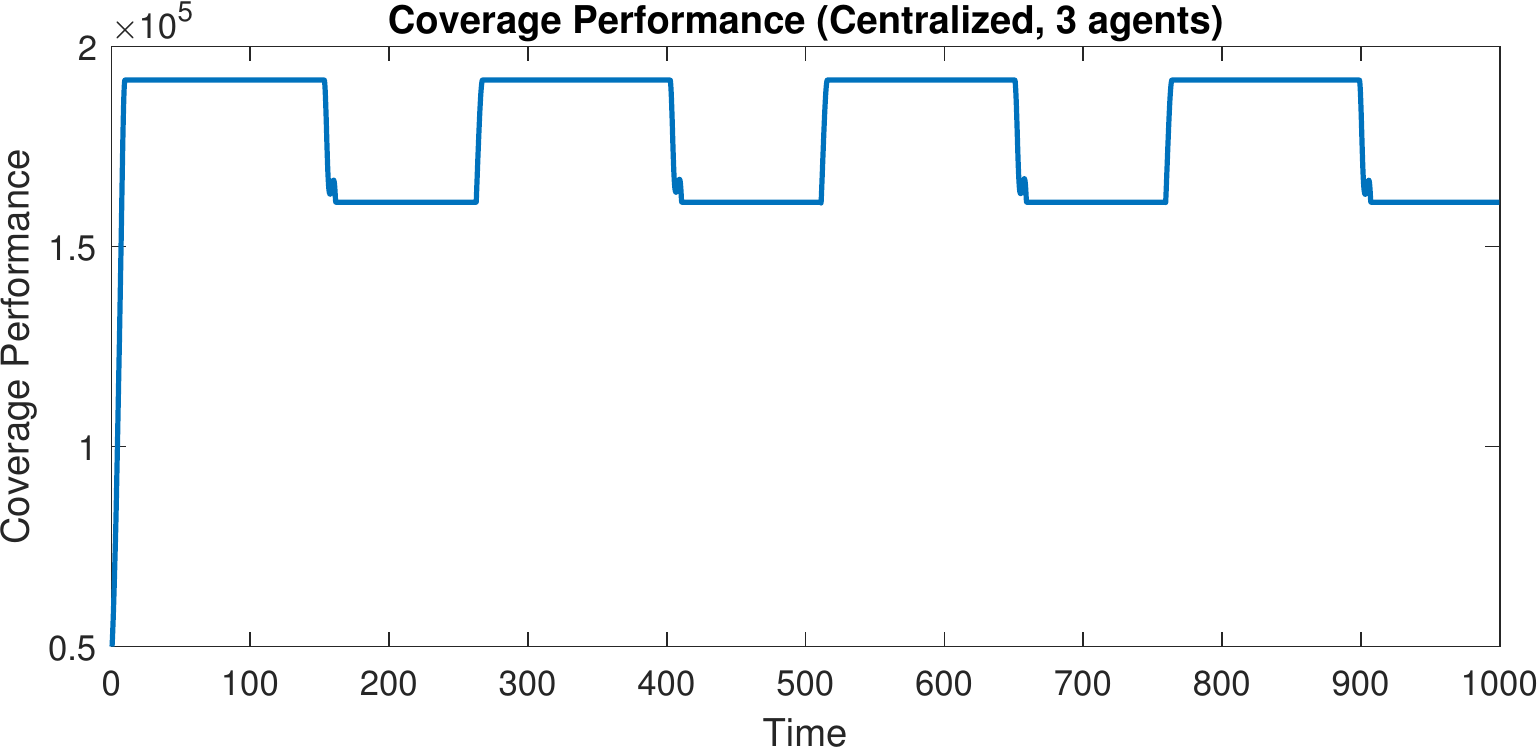}
\caption{Performance with Centralized Approach}
\label{f1_c}
\end{figure}

The coverage performance of the decentralized approach is computed using the approach proposed in~\cite{meng2018multi}. In the decentralized approach, agents may compete for the charging station. When this case occurs, the agent with lower priority has to turn off its sensing capability, therefore, performance may be significantly compromised. The coverage performance over time for the decentralized approach is shown in Fig.~\ref{f2}.

\begin{figure}[htbp]
\includegraphics[width=\linewidth]{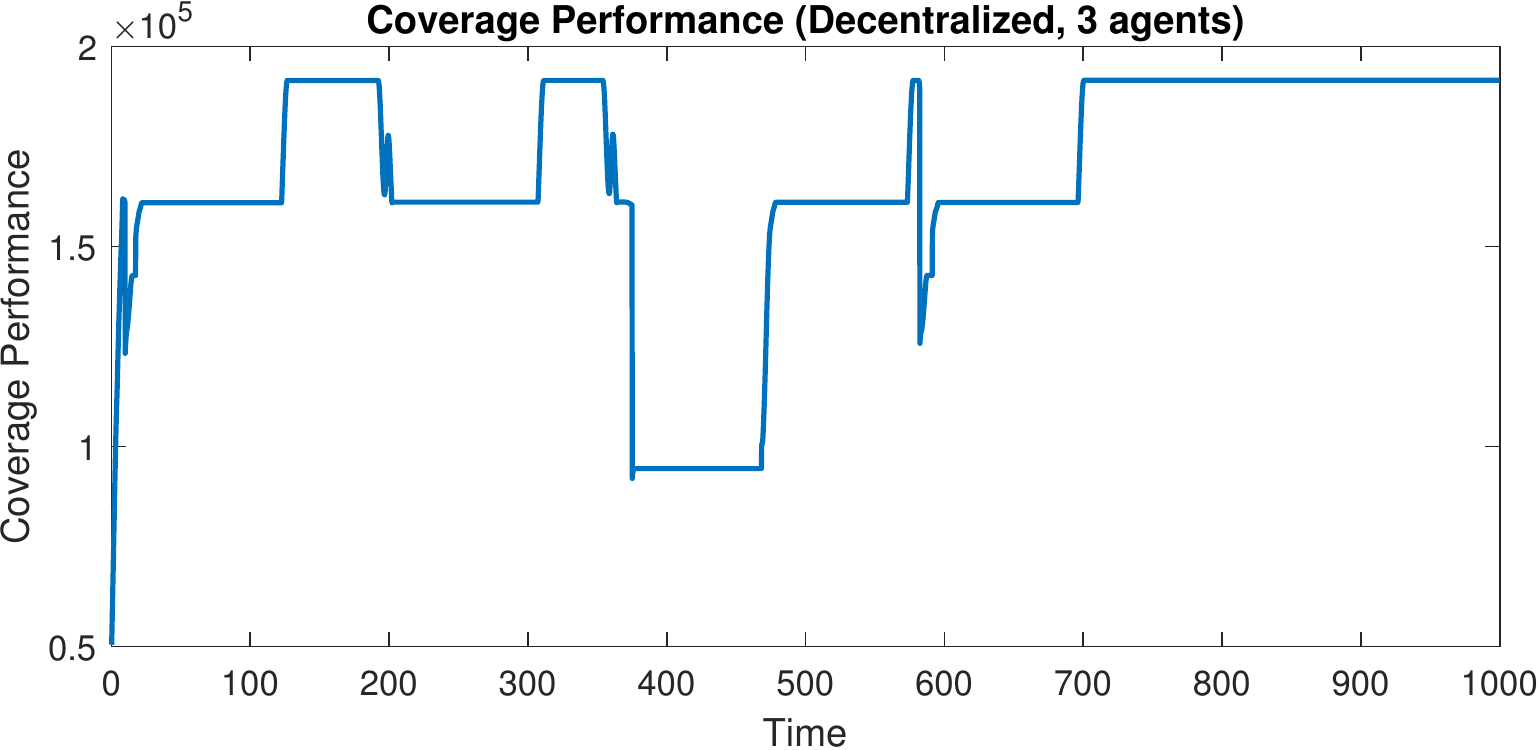}
\caption{Performance with Decentralized Approach}
\label{f2}
\end{figure}
The average coverage performance over a time period of 1000 seconds of the centralized and decentralized approaches is $177815$ and $166917$, respectively. The performance improvement is about $6.53\%$. Also note from both figures, the performance lower bound of the centralized approach is determined by the OCH formation. The bottom horizontal line in Fig.~\ref{f2} indicates that low priority agents turn off their sensing when competing for the charging station. The results show that both the average and the worst performance is significantly improved by the centralized approach.

\begin{figure}[htbp]
\includegraphics[width=\linewidth]{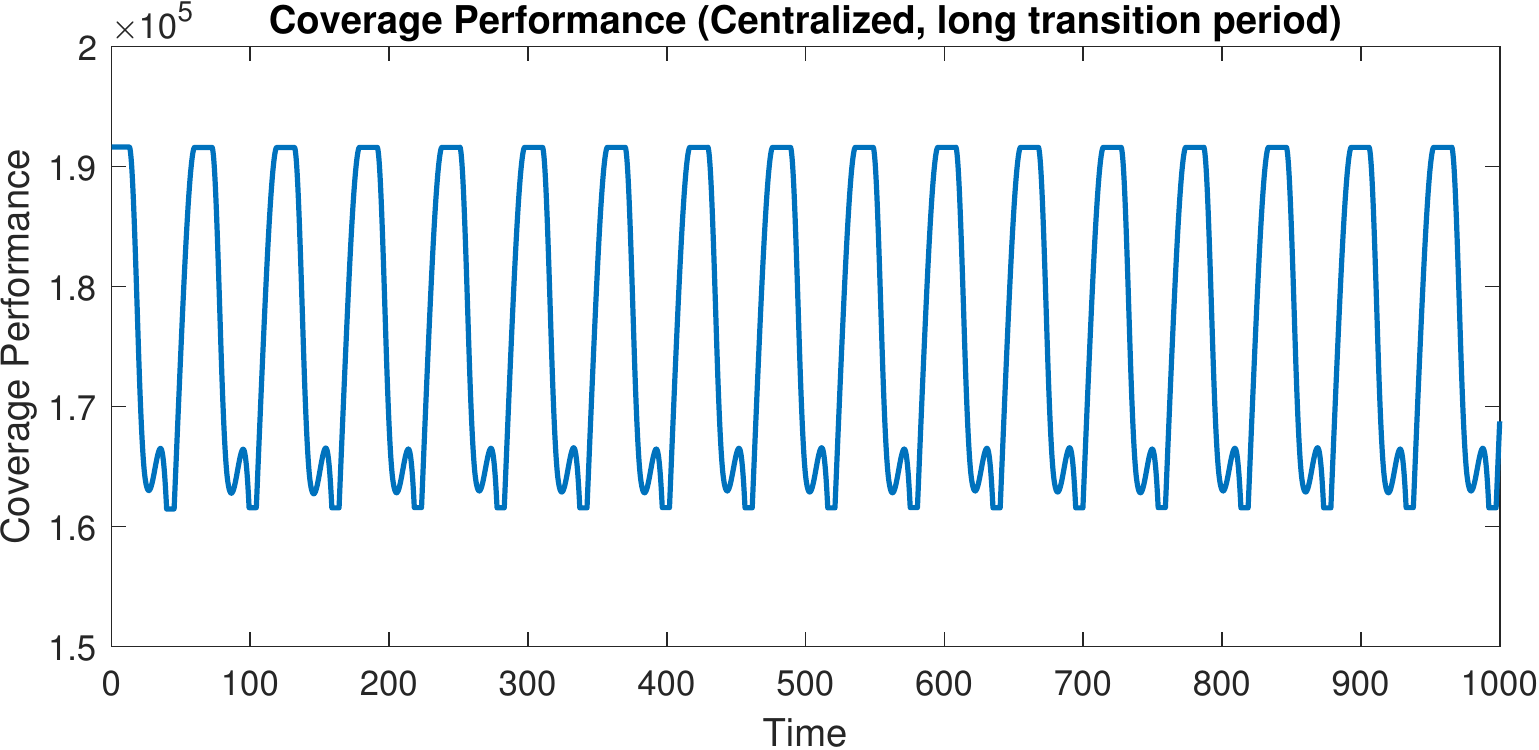}
\caption{Performance with Centralized Approach}
\label{f3_c}
\end{figure}

Another set of simulation is done with 6 agents. In this case, the parameters are chosen as $\alpha = 0.0005$, $\beta=0.0005$, $c=0.025$, and $\bar{v}=100$.
The coverage performance over time for the centralized approach and decentralized approach is depicted in Fig.~\ref{f3_c} and Fig.~\ref{f4}, respectively. The average coverage performance over time is $262946$ for the centralized approach and $253278$ for the decentralized approach. In this case, the average coverage performance improvement is $3.28\%$ by the centralized approach. When the number of agents increases, the centralized approach keeps a minimum coverage performance above $25000$. However, the performance is critically compromised for the decentralized approach when more agents compete for the charging stations, as shown in Fig.~\ref{f4}. 
\begin{figure}[htbp]
\includegraphics[width=\linewidth]{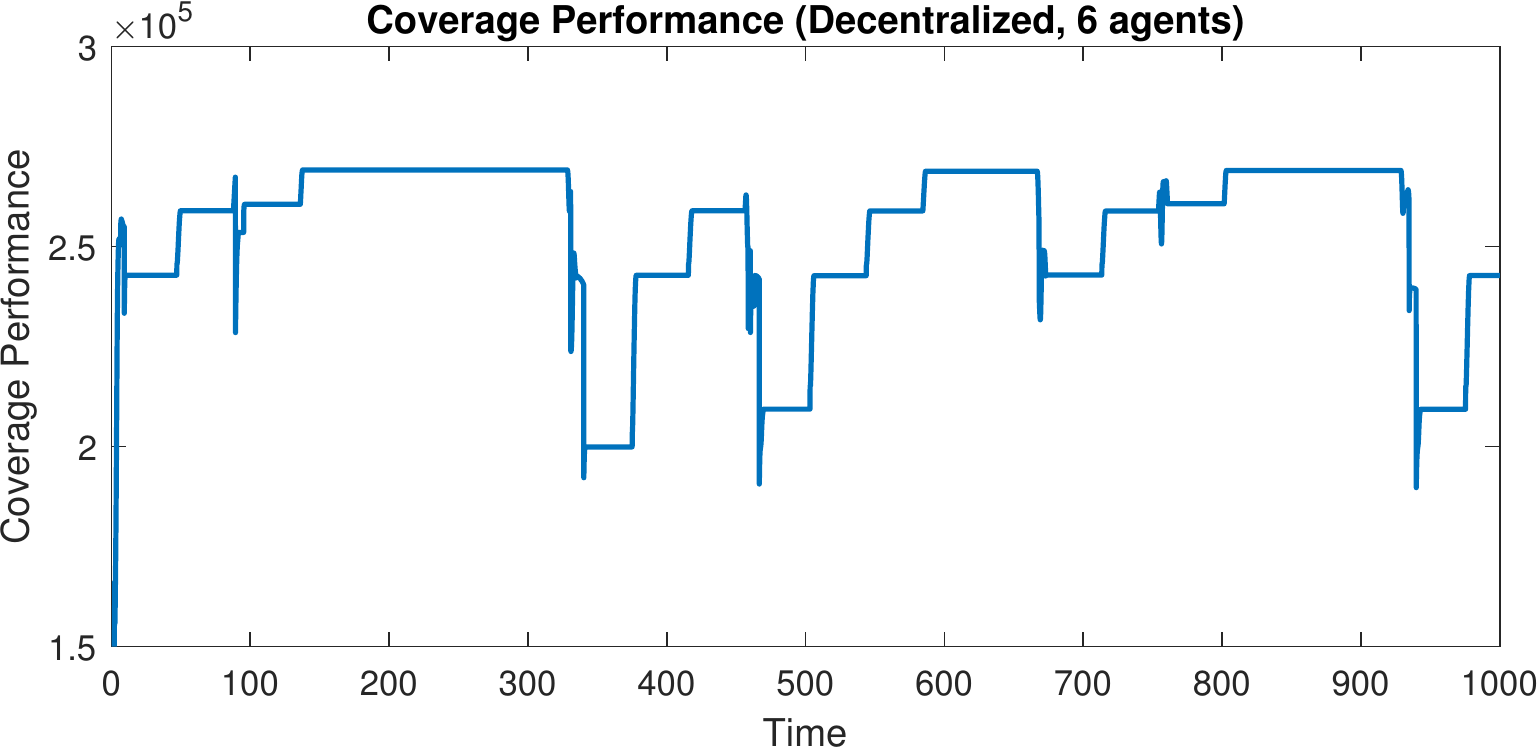}
\caption{Performance with Decentralized Approach}
\label{f4}
\end{figure}
\section{Conclusions}
In this paper, we propose a centralized near-optimal solution to the multi-agent coverage problem with energy constrained agents. The performance between the centralized approach and decentralized approach is compared. It shows that the centralized approach in general produces better average coverage performance than the decentralized approach. In addition, the performance gap between the OCV formations and the OCH formations of the centralized approach is much smaller than that of the decentralized approach. 
\bibliographystyle{IEEEtran}
\bibliography{sample}

\end{document}